\documentclass[11pt]{article}
\usepackage[utf8]{inputenc}

\usepackage{subfigure}
\usepackage{tikz}

\usepackage[utf8]{inputenc}

\usepackage[utf8]{inputenc}
\usepackage[english]{babel}
\usepackage{amsmath}
\usepackage{times}
\usepackage{graphicx}
\usepackage{float}
\usepackage[margin=1in]{geometry}
\usepackage{blindtext}
\usepackage{amssymb}
\usepackage{tikz}
\usepackage{amsthm}
\usepackage{mathtools}
\usepackage{textcmds}

\usepackage{comment}
\newtheorem{theorem}{Theorem}[section]
\newtheorem{corollary}[theorem]{Corollary}

\newtheorem{lemma}[theorem]{Lemma}
\newtheorem{conjecture}[theorem]{Conjecture}
\theoremstyle{remark}

\theoremstyle{definition}

\date{}

\begin{document}
\title{On a conjecture of spectral extremal problems\thanks{This research was partially supported by  the National Nature Science Foundation of China (Nos. 11871329, 11971298)}}
\author {Jing Wang, \, Liying Kang\thanks{\em Corresponding author.\newline  Email address: lykang@shu.edu.cn (L. Kang),  wj517062214@163.com (J. Wang),  xys16720018@163.com (Y. Xue)}, \, Yisai Xue \\
{\small Department of Mathematics, Shanghai University,
Shanghai 200444, P.R. China}}

\maketitle

\vspace{-0.5cm}

\begin{abstract}For a  simple graph $F$, let $\mathrm{Ex}(n, F)$ and $\mathrm{Ex_{sp}}(n,F)$ denote the set of graphs with the maximum number of edges and the set of graphs with the maximum spectral radius in an $n$-vertex graph without any copy of the graph $F$, respectively.
The Tur\'an graph $T_{n,r}$ is the complete $r$-partite graph on $n$ vertices where its part sizes are as equal as possible.
Cioab\u{a}, Desai and Tait [The spectral radius of graphs with no odd wheels, European J. Combin., 99 (2022) 103420]  posed the following conjecture: Let $F$ be any graph such that the graphs in $\mathrm{Ex}(n,F)$ are Tur\'{a}n graphs plus $O(1)$ edges. Then $\mathrm{Ex_{sp}}(n,F)\subset \mathrm{Ex}(n,F)$ for sufficiently large $n$.
In this paper we consider the graph $F$ such that the graphs in $\mathrm{Ex}(n, F)$ are obtained from $T_{n,r}$ by adding $O(1)$ edges, and prove that if $G$ has the maximum spectral radius among all $n$-vertex graphs not containing $F$, then $G$ is a member of $\mathrm{Ex}(n, F)$  for $n$ large enough. Then Cioab\u{a}, Desai and Tait's  conjecture is completely solved.
 \end{abstract}

{{\bf Key words:}   Spectral radius; Spectral extremal graph; Tur\'an graph. }

\section{Introduction}

Let $F$ be a simple graph. A graph $G$ is $F$-free if there is no subgraph of $G$ isomorphic to $F$. The Tur\'an type extremal problem is to determine the maximum number of edges in a graph on $n$ vertices that is $F$-free, and the maximum number of edges is called the {\sl Tur\'an number}, denoted by $\mathrm{ex}(n,F)$.
Such a graph with $\mathrm{ex}(n, F)$ edges is called an {\sl extremal graph} for $F$ and we denote by $\mathrm{Ex}(n, F)$ the set of all extremal graphs on $n$ vertices for $F$. The {\sl Tur\'{a}n graph} is the complete $r$-partite graph on $n$ vertices where each partite set has either $\lfloor\frac{n}{r}\rfloor$ or $\lceil\frac{n}{r}\rceil$ vertices and the edge set consists of all pairs joining distinct parts, denoted by $T_{n,r}$. The well-known Tur\'{a}n Theorem \cite{Turan41} states that the extremal graph corresponding to Tur\'{a}n number $\mathrm{ex}(n,K_{r+1})$ is $T_{n,r}$, i.e. $\mathrm{ex}(n,K_{r+1})=e(T_{n,r})$. Erd\H{o}s, Stone and Simonovits \cite{ES46,ES66} presented the following result
 \begin{equation} \label{eqESS}
  \mathrm{ex}(n,F) = \left( 1- \frac{1}{\chi (F) -1}  \right)
  \frac{n^2}{2} + o(n^2),
  \end{equation}
 where $\chi (F)$ is the vertex-chromatic number of $F$. There are lots of researches on Tur\'{a}n type extremal problems (such as \cite{Bollobas78,Erdos95,Chen03,Keevash11,Sim13,FS13}).

In this paper we focus on spectral analogues of the Tur\'{a}n type problem for graphs, which was proposed by Nikiforov \cite{NikiforovLAA10}. The spectral Tur\'{a}n type problem is to determine the maximum spectral radius instead of the number of edges among all $n$-vertex $F$-free graphs. The graph which attains the maximum spectral radius is called a spectral extremal graph. We denote by $\mathrm{Ex_{sp}}(n,F)$ the set of all spectral extremal graphs for $F$. Researches of the spectral Tur\'{a}n type problem have drawn increasing extensive intersect (see \cite{NikiforovTuran,BG09,FiedlerNikif,NikifSurvey,ZW12}). Nikiforov \cite{Nikiforov07} showed that if $G$ is a $K_{r+1}$-free graph on $n$ vertices, then $\lambda (G)\le \lambda (T_{n,r})$, with equality if and only if $G=T_{n,r}$.  This implies that if $G$ attains the maximum spectral radius over all $n$-vertex $K_{r+1}$-free graphs for sufficiently large $n$, then $G\in \mathrm{Ex}(n,K_{r+1})$. Cioab\u{a}, Feng,  Tait and  Zhang \cite{CFTZ20} proved that the spectral extremal graph for $F_k$ belongs to $\mathrm{Ex}(n,F_k)$, where $F_k$ is the graph consisting of $k$ triangles which intersect in exactly one common vertex. In addition,   Chen, Gould, Pfender and Wei \cite{Chen03}
 proved that
 $\mathrm{ex}(n,F_{k,r+1})=e(T_{n,r})+O(1)$, where $F_{k,r+1}$ is the graph consisting of $k$ copies of $K_{r+1}$ which intersect in a single vertex. Naturally, Cioab\u{a}, Desai and Tait \cite{CDT21} raised the following conjecture.
\begin{conjecture}[Cioab\u{a} et al. \cite{CDT21}] \label{conj}
Let $F$ be any graph such that the graphs in $\mathrm{Ex}(n,F)$ are Tur\'{a}n graphs plus $O(1)$ edges. Then $\mathrm{Ex_{sp}}(n,F)\subset \mathrm{Ex}(n,F)$ for sufficiently large $n$.
\end{conjecture}
The results of Nikiforov \cite{Nikiforov07}, Cioab\u{a}, Feng,  Tait and  Zhang \cite{CFTZ20}, Li and Peng \cite{Yongtao21}, and  Desai, Kang, Li, Ni, Tait and Wang \cite{DKLNTW} tell us that  Conjecture \ref{conj} holds for $K_{r+1}$, $F_k$, $H_{s,k}$ and $F_{k,r}$, where $H_{s,k}$ is the graph defined by intersecting $s$ triangles and $k$ odd cycles of length at least $5$ in exactly one common vertex.  In this paper, we shall prove the following theorem which confirms Conjecture \ref{conj}.

\begin{theorem} \label{main result}
Let $r\geq 2$ be an integer, and $F$ be any graph such that the graphs in $\mathrm{Ex}(n,F)$ are obtained from $T_{n,r}$ by adding $O(1)$ edges. For sufficiently large $n$, if $G$ has the maximal spectral radius over all $n$-vertex $F$-free graphs, then
$$G \in \mathrm{Ex}(n, F).$$
\end{theorem}

\section{Notation and Preliminaries}\label{sec: lemmas}

In this section we introduce some notation and give the preparatory lemmas.

Let $G=(V(G),E(G))$ be a simple graph with vertex set $V(G)$ and edge set $E(G)$.  If $u,v\in V(G)$, and $uv\in E(G)$, then $u$ and $v$ are said to be {\sl adjacent}. For a vertex $v\in V(G)$, the {\sl neighborhood} $N_{G}(v)$ (or simply $N(v)$) of $v$ is $\{u|\ uv\in E(G)\}$, and the {\sl degree} $d_{G}(v)$ (or simply $d(v)$) of $v$ is $|N_{G}(v)|$. The minimum and maximum degrees are denoted by $\delta(G)$ and $\Delta(G)$, respectively. For $S\subseteq V(G)$ and $v\in V(G)$, let $d_{S}(v)=|N_S(v)|=|N_{G}(v)\cap S|$. 
For $V_1,V_2 \subseteq V(G)$, $e(V_1,V_2)$ denotes the number of edges of $G$ between  $V_1$ and  $V_2$. For $S\subseteq V(G)$, denote by $G\setminus S$ the graph
obtained from $G$ by deleting all vertices of $S$ and the incident edges. Denote by $G[S]$ the graph  induced by $S$ whose vertex set is $S$ and whose edge set consists of all edges of $G$ which have both ends in $S$.

Let $G$ be a simple graph with $n$ vertices. The {\sl adjacent matrix} of $G$ is $A(G)=(a_{ij})_{n\times n}$ with $a_{ij}=1$ if $ij\in E(G)$, and $a_{ij}=0$ otherwise. The {\sl spectral radius} of $G$ is the largest eigenvalue of $A(G)$, denoted by $\lambda(G)$. Let $G_1,\ldots, G_s$ be the components of $G$, then $\lambda(G)=\max\{ \lambda(G_i)|\ i\in [s]\}$. 
For a connected graph $G$, let $\mathbf{x}=(x_1,\ldots,x_n)^{\mathrm{T}}$ be an eigenvector of $A(G)$ corresponding to $\lambda(G)$. Then $\mathbf{x}$ is a positive real vector, and
\begin{equation}\label{eigenequation}
\lambda(G)x_i=\sum_{ij\in E(G)}x_j, \text{ for any } i\in [n].
\end{equation}
The following Rayleigh quotient equation is very useful:
\begin{equation}\label{Rayleigh}
\lambda(G)=\max_{\mathbf{x}\in \mathbb{R}^{n}_{+}}\frac{\mathbf{x}^{\mathrm{T}}A(G)\mathbf{x}}{\mathbf{x}^{\mathrm{T}}\mathbf{x}}=\max_{\mathbf{x}\in \mathbb{R}^{n}_{+}}\frac{2\sum_{ij\in E(G)}x_ix_j}{\mathbf{x}^{\mathrm{T}}\mathbf{x}}.
\end{equation}

We have the following two lemmas from Zhan \cite{Zhan13}.
\begin{lemma}[Zhan \cite{Zhan13}]\label{zhan1}
Let $A$ and $B$ be two nonnegative square matrices. If $B<A$ and $A$ is irreducible, then $\lambda(B)< \lambda(A)$.
\end{lemma}

\begin{lemma}[Zhan \cite{Zhan13}]\label{zhan2}
Let $A$ be a nonnegative square matrix. If $B$ is a principal submatrix of $A$, then $\lambda(B)\leq \lambda(A)$. If $A$ is irreducible and $B$ is a proper principal submatrix of $A$, then $\lambda(B)< \lambda(A)$.
\end{lemma}
Let $A(G)$ be the adjacent matrix of graph $G$. Then $G$ is connected if and only if $A(G)$ is irreducible. Combining with Lemmas \ref{zhan1} and \ref{zhan2}, we have the following result.

\begin{lemma}\label{subgraph}
Let $G$ be a connected graph. If $G'$ is a proper subgraph of $G$, then $\lambda(G')< \lambda(G)$.
\end{lemma}

Recall the classical stability theorem proved by Erd\H{o}s \cite{Erdos1966,Erdos1968} and Simonovits \cite{Sim66}:
\begin{lemma}[Erd\H{o}s \cite{Erdos1966,Erdos1968}, Simonovits \cite{Sim66}]\label{edgestability}
For every $r\geq 2, \varepsilon >0$, and $(r+1)$-chromatic graph $F$, there exists $\delta>0$ such that if a graph $G$ of order $n$ satisfies $e(G)> (1-\frac{1}{r}-\delta)\frac{n^2}{2}$, then either $G$ contains $F$, or $G$ differs form $T_{n,r}$ in at most $\varepsilon n^2$ edges.
\end{lemma}

Write $K_r(n_1,\ldots,n_r)$ for the complete $r$-partite graph with classes of sizes $n_1,\ldots, n_r$. Nikiforov \cite{Niki09JGT} proved the spectral version of Stability Lemma.

\begin{lemma}[Nikiforov \cite{Niki09JGT}]  \label{lemniki}
Let $r\ge 2, 1/\ln n < c < r^{-8(r+21)(r+1)}, 0< \varepsilon < 2^{-36}r^{-24}$ and $G$ be a graph on $n$ vertices. If $\lambda (G) > (1- \frac{1}{r} - \varepsilon )n$, then one of the following statements holds: \\
(a) $G$ contains a $K_{r+1}(\lfloor c\ln n\rfloor , \ldots ,\lfloor c\ln n\rfloor,\lceil n^{1-\sqrt{c}}\rceil)$; \\
(b) $G$ differs from $T_{n,r}$ in fewer than $(\varepsilon^{1/4} + c^{1/(8r+8)})n^2$ edges.
\end{lemma}

From the above theorem, one can easily get the following result.

\begin{corollary}  \label{coro22}
Let $F$ be a graph with chromatic number $\chi (F)=r+1$. For every $\varepsilon >0$, there exist $\delta >0$ and $n_0$ such that if  $G$ is an $F$-free graph on $n\ge n_0$ vertices  with $\lambda (G) \ge (1- \frac{1}{r} -\delta )n$, then $G$ can be obtained from $T_{n,r}$ by adding and deleting at most $\varepsilon n^2$ edges.
\end{corollary}

For $K_{r}(n_1,n_2,\ldots,n_{r})$, let $n=\sum_{i=1}^{r}n_i$. For convenience, we assume that $n_1\geq n_2\geq \ldots \geq n_{r}>0$. It is well-known \cite[p. 74]{CDS1980} or \cite{Delorme}
that the characteristic polynomial of $K_{r}(n_1,n_2,\ldots,n_{r})$ is given as
$$\phi(K_{r}(n_1,n_2,\ldots,n_{r}), x)=x^{n-r}\left(1-\sum_{i=1}^{r}\frac{n_i}{x+n_i}\right)\prod_{j=1}^{r}(x+n_j).$$
So the spectral radius $\lambda(K_{r}(n_1,n_2,\ldots,n_{r}))$ satisfies the following equation:
\begin{equation}
\sum_{i=1}^{r}\frac{n_i}{\lambda(K_{r}(n_1,n_2,\ldots,n_{r}))+n_{i}}=1 \label{eq0}
\end{equation}

Feng, Li and Zhang \cite[Theorem 2.1]{FLZ2007} proved the following lemma, which can also be seen in Stevanovi\'{c}, Gutnam and Rehman \cite{Stevanovicetal}.

\begin{lemma}[Feng et al. \cite{FLZ2007}, Stevanovi\'{c} et al. \cite{Stevanovicetal}]\label{rpartitegraph}
If $n_i-n_j\geq 2$, then $$\lambda(K_{r}(n_1,\ldots,n_i-1,\ldots,n_j+1,\ldots,n_{r}))>\lambda(K_{r}(n_1,\ldots,n_i,\ldots,n_j,\ldots,n_{r})).$$
\end{lemma}

The following lemma was given in \cite{CFTZ20}.

\begin{lemma}[Cioab\u{a} et al. \cite{CFTZ20}] \label{inclusion exclusion lemma}
Let $A_1, \ldots, A_p$ be finite sets. Then
\[|A_1 \cap \ldots \cap A_p| \geq \sum_{i=1}^p |A_i| - (p-1)\bigg| \bigcup_{i=1}^p A_i \bigg|.\]
\end{lemma}

 \section{Proof of Theorem \ref{main result}}

Let $F$ be any graph such that the graphs in $\mathrm{Ex}(n,F)$ are Tur\'{a}n graphs plus $O(1)$ edges. We may assume that
the graphs in $\mathrm{Ex}(n,F)$ are obtained from $T_{n,r}$ by adding $a$ edges. Then $\mathrm{ex}(n,F)=e(T_{n,r})+a$, which implies that $\chi(F)=r+1$ by  (\ref{eqESS}) and the fact $(1-\frac{1}{r})\frac{n^2}{2}-\frac{r}{8}\leq e(T_{n,r})\leq (1-\frac{1}{r})\frac{n^2}{2}$. 
In the sequel, we always assume that $G$ is a graph on $n$ vertices containing no $F$ as a subgraph and attaining the maximum spectral radius. The aim of this section is to prove that $G$ is obtained from $T_{n,r}$ by adding $a$ edges for $n$ large enough.

The sketch of our proof is as follows: Firstly, we give a lower bound on $\lambda(G)$, and determine a partition $V(G)=V_1 \cup \cdots \cup V_{r}$ such that $\sum_{1\leq i<j\leq r}e(V_i,V_j)$ attains the maximum by using the spectral version of Stability Lemma. Then we show that any vertex except at most $2a$ vertices in  $V_i$ is adjacent to all vertices in $V_j$ for any $i,j\in [r]$ and $j\not=i$. Next, we prove that all vertices have eigenvector entry very close to the maximum entry and show that the partition is balanced. Finally, we prove $e(G)=\mathrm{ex}(n,F)$ by contradiction.

\begin{lemma}
$G$ is connected.
\end{lemma}

\begin{proof}
Suppose to the contrary that $G$ is not connected. Assume
$G_1,\ldots,G_s$ are the components of $G$ and $\lambda(G_1)=\max\{\lambda(G_i)|\ i\in [s]\}$, then $\lambda(G)=\lambda(G_1)$ and
$|V(G_1)|\leq n-1$. For any vertex $u\in V(G_1)$, let $G'$ be the graph obtained from $G_1$ by adding a pendent edge $uv$ at $u$ and $n-1-|V(G_1)|$ isolated vertices. Then $\lambda(G')>\lambda(G_1)=\lambda(G)$ by Lemma \ref{subgraph}. This implies that $G'$ contains a copy of  $F$ as a subgraph, denote it as $F_1$, then $uv$ is an edge of $F_1$. Next, we claim that $d_{G_1}(u)<|V(F)|$. Otherwise, $d_{G_1}(u)\geq |V(F)|$. Then there exists a vertex $w\in N_{G_1}(u)$ and $w\notin V(F_1)$. Then $F_1-uv+uw$ is a copy of $F$ in $G_1$, this is a contraction.
Due to the arbitrary of vertex $u$, $\Delta(G_1)<|V(F)|$.
Thus $\lambda(G)=\lambda(G_1)\leq \Delta(G_1)<|V(F)|<\lambda(T_{n,r})$ and this contradicts the fact that $G$ has the maximum spectral radius among all $n$-vertex $F$-free graphs as $T_{n,r}$ is $F$-free. Therefore, $G$ is connected.
\end{proof}

In the following, let $\lambda(G)$ be the spectral radius of $G$, $\mathbf{x}$ be a positive eigenvector corresponding to $\lambda(G)$ with $\max\{x_i|\ i\in V(G)\}=1$. Without loss of generality, we assume that $x_z = 1$. If there are multiple such vertices, we choose and fix $z$ arbitrarily among them.

\begin{lemma} \label{lem32}
$$\lambda(G) \ge \bigg(1- \frac{1}{r}\bigg)n - \frac{r}{4n}+\frac{2a}{n}.$$
\end{lemma}

\begin{proof}
  Let $H$ be an $F$-free graph on $n$ vertices with maximum number of edges.
  Since $G$ attains the maximum spectral radius over all $n$-vertex $F$-free graphs, and $\mathrm{ex}(n,F)=e(T_{n,r})+a$, by the Rayleigh quotient equation, we have
\begin{equation*}
\lambda (G) \geq \lambda (H) \geq \frac{\mathbf{1}^T A(H) \mathbf{1}}{\mathbf{1}^T\mathbf{1}}
= \frac{ 2( e(T_{n,r})+ a)}{n}\geq \frac{2}{n}\left(\left(1-\frac{1}{r}\right)\frac{n^2}{2}-\frac{r}{8}+a\right)\geq \left(1- \frac{1}{r}\right)n - \frac{r}{4n}+\frac{2a}{n}.
\end{equation*}
\end{proof}

 Let $\ell$ be an integer  satisfying $\ell\gg \max\{a,|V(F)|\}$.

\begin{lemma}
\label{approximate structure}
For every  $\epsilon > 0$, there exists an integer $n_0$ such that if $n\geq n_0$,  then
$$e(G) \geq e(T_{n,r}) - \epsilon n^2.$$
Furthermore,
$G$ has a partition $V(G) = V_1 \cup \ldots \cup V_{r}$ such that $\sum_{1\leq i<j\leq r}e(V_i,V_j)$ attains the maximum, and
\[ \sum_{i=1}^{r} e(V_i) \leq \epsilon n^2, \]
and for each $i \in [r]$,
 \[ \left(\frac{1}{r} - 3\sqrt{\epsilon}\right)n < |V_i| < \left(\frac{1}{r} + 3\sqrt{\epsilon}\right)n. \]
\end{lemma}

\begin{proof} From Lemma \ref{lem32} and Corollary \ref{coro22}, it follows that $G$ is obtained from $T_{n,r}$ by adding or deleting at most $\epsilon n^2$ edges for large enough $n$. Then there is a partition of $V(G) = U_1 \cup \ldots \cup U_{r}$ with $\sum_{i=1}^{r} e(U_i) \leq \epsilon n^2$, $\sum_{1\leq i<j\leq r}e(U_i,U_j)\geq  e(T_{n,r}) - \epsilon n^2$ and $\lfloor\frac{n}{r}\rfloor \leq |U_i| \leq \lceil\frac{n}{r}\rceil$ for each $i \in [r]$. So  $e(G) \geq e(T_{n,r}) - \epsilon n^2$. Furthermore,
$G$ has a  partition $V = V_1 \cup \ldots \cup V_{r}$ such that $\sum_{1\leq i<j\leq r}e(V_i,V_j)$ attains the maximum. In this case, $\sum_{i=1}^{r} e(V_i) \leq \sum_{i=1}^{r} e(U_i) \leq \epsilon n^2$ and $\sum_{1\leq i<j\leq r}e(V_i,V_j)\geq\sum_{1\leq i<j\leq r}e(U_i,U_j)\geq e(T_{n,r}) - \epsilon n^2$. Let $s=\max\left\{\left||V_j|-\frac{n}{r}\right|, j\in [r]\right\}$.  Without loss of generality, we assume $\left||V_1|-\frac{n}{r}\right|=s$. Then
\begin{eqnarray*}
e(G)&\leq& \sum_{1\leq i<j\leq r}|V_i||V_j|+\sum_{i=1}^{r}e(V_i)\nonumber\\[2mm]
&\leq & |V_1|(n-|V_1|)+ \sum_{2\leq i<j\leq r}|V_i||V_j| +\epsilon n^2\nonumber\\
&=& |V_1|(n-|V_1|)+ \frac{1}{2}\Big((\sum_{j=2}^{r}|V_j|)^2-\sum_{j=2}^{r}|V_j|^2\Big) +\epsilon n^2\nonumber\\[2mm]
&\leq& |V_1|(n-|V_1|)+\frac{1}{2}(n-|V_1|)^2-\frac{1}{2(r-1)}(n-|V_1|)^2+\epsilon n^2\\[2mm]
&<&-\frac{r}{2(r-1)}s^2+\frac{r-1}{2r}n^2+\epsilon n^2,
\end{eqnarray*}
where  the last second inequality  holds by H\"{o}lder's inequality, and  the last inequality holds since $\left||V_1|-\frac{n}{r}\right|=s$. On the other hand, $$e(G)\geq e(T_{n,r}) - \epsilon n^2\geq \left(1-\frac{1}{r}\right) \frac{n^2}{2} - \frac{r}{8}-\epsilon n^2> \frac{r-1}{2r}n^2-2\epsilon n^2,$$
as $n$ is large enough. Therefore,
$\frac{r}{2(r-1)}s^2<3\epsilon n^2$, which implies that $s<\sqrt{\frac{6(r-1)\epsilon}{r}n^2}<\sqrt{6\epsilon}n< 3\sqrt{\epsilon}n$. The proof is completed.
\end{proof}

\begin{lemma}
\label{lem size of W and L}
Let $\theta>0$ and $\epsilon>0$ be sufficiently small constants with $\theta< \frac{1}{100r^5\ell}$ and $2\epsilon< \theta^3$.
We denote
\begin{equation}
\label{eqn defn of W}
    W := \cup_{i = 1}^{r} \{v \in V_i |\ d_{V_i}(v) \geq 2\theta n \},
\end{equation}
and
\begin{equation}
\label{eqn defn of L}
    L := \bigg\{v \in V(G) |\ d(v) \leq \bigg( 1 - \frac{1}{r} - 3r\epsilon^{\frac{1}{3}} \bigg)n\bigg\}.
\end{equation}
Then $|L| \leq \epsilon^{\frac{1}{3}} n$ and $W \subseteq L$.

\end{lemma}

\begin{proof}
We first prove the following claims.
\vskip 2mm
\noindent{\bfseries Claim 1.}  $|W|<\theta n $

\begin{proof}
It follows  from Lemma \ref{approximate structure} that $\sum_{i=1}^{r}e(V_i)\leq \epsilon n^2$. On the other hand, let $W_i : = W \cap V_i$ for all $i \in [r]$. Then
\[2 e(V_i) = \sum_{u\in V_i}d_{V_i}(u) \geq \sum_{u\in W_i}d_{V_i}(u) \geq 2|W_i|\theta n\]
Thus \[\sum_{i=1}^{r}e(V_i) \geq \sum_{i=1}^{r} |W_i|\theta n = |W|\theta n.\]
Therefore, we have that $|W|\theta n \leq \epsilon n^2$. This proves that $|W|\leq \frac{\epsilon n}{\theta} < \theta n.$
\end{proof}

\noindent{\bfseries Claim 2.} $|L| \leq \epsilon^{\frac{1}{3}} n$.

\begin{proof}
Suppose to the contrary that $|L| > \epsilon^{\frac{1}{3}} n$. Then there exists a subset $L' \subseteq L$ with $|L'| = \lfloor \epsilon^{\frac{1}{3}} n \rfloor$.
Therefore,
\begin{equation*}
    \begin{split}
 e(G [V\setminus L']) \geq e(G) - \sum_{v \in L'}d(v)
 & \geq e(T_{n,r}) - \epsilon n^2 - \epsilon^{\frac{1}{3}} n^2 \bigg(1 - \frac{1}{r} - 3r\epsilon^{\frac{1}{3}}\bigg)\\
& > \frac{(n - \lfloor \epsilon^{\frac{1}{3}} n \rfloor)^2}{2} \bigg(1-\frac{1}{r}\bigg) + a \\
& \geq e(T_{n',r}) + a= \mathrm{ex}(n',F),
    \end{split}
\end{equation*}
where $n'=n - \lfloor \epsilon^{\frac{1}{3}} n \rfloor$ and $n$ is large enough. However, $ e(G [V\setminus L']) > \mathrm{ex}(n',F)$ implies that $G [V\setminus L']$ contains an $F$, which contradicts that $G$ is $F$-free.
\end{proof}

\vskip 2mm
Next, we prove that  $W \subseteq L$.
Otherwise, there exists a vertex $u_0 \in W$ and $u_0 \notin L$. Without loss of generality, let $u_0\in V_1$. Since
 $V(G) = V_1 \cup \ldots \cup V_{r}$ is the   partition
  such that $\sum_{1\leq i<j\leq r}e(V_i,V_j)$ attains the maximum,  $d_{V_1}(u_0)\le d_{V_i}(u_0)$ for each $i\in [2,r]$. 
Thus $d(u_0)\ge rd_{V_1}(u_0)$, that is $d_{V_1}(u_0)\leq \frac{1}{r}d(u_0)$. On the other hand,
since $u_0 \not\in L$, we get $d(u_0)> (1-\frac{1}{r}-3r\epsilon^{\frac{1}{3}} )n$. Thus
\begin{align}
 d_{V_2}(u_0)\nonumber
 &\ge d(u_0) - d_{V_1}(u_0) - (r-2)\left(\frac{1}{r} + 3\sqrt{\epsilon} \right)n\\\nonumber
 & \ge \left(1- \frac{1}{r}\right) d(u_0) - (r-2)\left(\frac{1}{r} + 3\sqrt{\epsilon} \right)n\\\label{degreeu0}
 & > \frac{n}{r^2} - 3(r-1) \epsilon^{\frac{1}{3}} n - 3(r-2) \sqrt{\epsilon} n\\\nonumber
 & >  \frac{n}{r^2}-6r\epsilon^{\frac{1}{3}} n.
\end{align}
Recall from Claim 1 and Claim 2 that
$  |W| <{\theta} n$ and $ |L|\le \epsilon^{\frac{1}{3}}n$, hence, for any $i\in [r]$ and sufficiently large $n$, we have
\begin{equation*}
  |V_i\setminus (W\cup L)|
  \ge \left(\frac{1}{r}- 3\sqrt{\epsilon} \right)n-\theta n -\epsilon^{\frac{1}{3}} n \ge  \ell.
\end{equation*}

We claim that $u_0$ is adjacent to at most $a$ vertices in $V_1\setminus (W\cup L)$.
Otherwise, let $u_{1,1}, u_{1,2}, \ldots, u_{1,a+1}$ be the neighbors of $u_0$ in $V_1\setminus (W\cup L)$.
Let $u_{1,a+2}, \ldots, u_{1,\ell}$ be another $\ell-a-1$ vertices in $V_1\setminus (W\cup L)$. For any $j\in [\ell]$, since $u_{1,j}\not\in L$ and $u_{1,j}\not\in W$,  we have $ d(u_{1,j})> \left(1-\frac{1}{r}-3r\epsilon^{\frac{1}{3}} \right)n$, and $d_{V_1}(u_{1,j})< 2\theta n$. Thus,
\begin{align}
d_{V_2}(u_{1,j})\nonumber
& \ge d(u_{1,j})-d_{V_1}(u_{1,j})- (r-2) \left(\frac{1}{r}+ 3\sqrt{\epsilon} \right)n  \\\label{degreeu}
& > \frac{n}{r} - 3r\epsilon^{\frac{1}{3}} n - 2\theta n - 3(r-2)\sqrt{\epsilon} n  \\\nonumber
& > \frac{n}{r} - 6r\epsilon^{\frac{1}{3}} n - 2\theta n .
\end{align}
By Lemma~\ref{inclusion exclusion lemma}, we consider the common neighbors of $u_0,u_{1,1}, \ldots ,u_{1,\ell}$ in $V_2$,
\begin{eqnarray*}
&& \left| N_{V_2}(u_0) \cap N_{V_2}(u_{1,1})\cap  \cdots \cap N_{V_2}(u_{1,\ell}) \setminus (W\cup L) \right| \\
 &\ge& d_{V_2}(u_0) + \sum_{j=1}^{\ell} d_{V_2}(u_{1,j})- \ell \left| V_2 \right| - |W| - |L| \\
 &>&  \frac{n}{r^2}-6r\epsilon^{\frac{1}{3}} n + \ell\left(\frac{n}{r} - 6r\epsilon^{\frac{1}{3}} n - 2\theta n\right) - \ell\left(\frac{n}{r}+3\sqrt{\epsilon}n\right) - \theta n - \epsilon^{\frac{1}{3}} n\\
& > & \frac{n}{r^2} - 16r\ell\epsilon^{\frac{1}{3}} n - (2\ell+1)\theta n   > \ell,
\end{eqnarray*}
for sufficiently large $n$. This implies that there exist $\ell$ vertices $u_{2,1},u_{2,2}, \ldots ,u_{2,\ell}$  in $V_2\setminus(W\cup L)$ such that $\{u_0,u_{1,1}, \ldots, u_{1,\ell}\}$ and $\{u_{2,1}, \ldots, u_{2,\ell}\}$ induce a complete bipartite graph. For an integer $s$ with $2\leq s\leq r-1$, suppose that for any $1\leq i\leq s$, there exist $u_{i,1},u_{i,2}, \ldots ,u_{i,\ell}\in V_i\setminus (W\cup L)$ such that $\{u_0,u_{1,1}, \ldots, u_{1,\ell}\}$, $\{u_{2,1}, \ldots, u_{2,\ell}\}$, $\ldots$, $\{u_{s,1}, \ldots, u_{s,\ell}\}$ induce a complete $s$-partite graph. We next consider the common neighbors of these vertices in $V_{s+1}$. Similarly, by  (\ref{degreeu0}) and (\ref{degreeu}), we get that for each $i\in [s]$ and $j\in [\ell]$,
\begin{align*}
d_{V_{s+1}}(u_0)
 & >  \frac{n}{r^2}-6r\epsilon^{\frac{1}{3}} n,
\end{align*}
and
\begin{align*}
d_{V_{s+1}}(u_{i,j})
& > \frac{n}{r} - 6r\epsilon^{\frac{1}{3}} n - 2\theta n .
\end{align*}
By Lemma~\ref{inclusion exclusion lemma} again, we can obtain
\begin{eqnarray*}
&& \left| N_{V_{s+1}}(u_0) \cap \left(\cap_{i\in [s],j\in [\ell]}N_{V_{s+1}}(u_{i,j}) \right) \setminus (W\cup L) \right| \\
&\ge& d_{V_{s+1}}(u_0) + \sum_{i\in [s], j\in [\ell]} d_{V_{s+1}}(u_{i,j}) - s\ell \left| V_{s+1} \right| - |W| - |L| \\
&>& \frac{n}{r^2}-6r\epsilon^{\frac{1}{3}} n + s\ell\left(\frac{n}{r} - 6r\epsilon^{\frac{1}{3}} n - 2\theta n\right) - s\ell\left(\frac{n}{r}+3\sqrt{\epsilon}n\right) - \theta n - \epsilon^{\frac{1}{3}} n\\
& > & \frac{n}{r^2} - 16sr\ell\epsilon^{\frac{1}{3}} n - (2s\ell+1)\theta n   > \ell,
\end{eqnarray*}
where $n$ is sufficiently large. Hence there exist $\ell$ vertices $u_{s+1,1},u_{s+1,2}, \ldots ,u_{s+1,\ell}\in V_{s+1}\setminus (W\cup L)$ such that $\{u_0,u_{1,1},\ldots,u_{1,\ell}\}$, $\ldots$, $\{u_{s+1,1}$, $\ldots$, $u_{s+1,\ell}\}$ induce a complete $(s+1)$-partite graph. Thus, for each $i \in [2,r]$, there exist $u_{i,1},u_{i,2}, \ldots ,u_{i,\ell}$ in $V_i \setminus (W\cup L)$ such that
$\{u_0,u_{1,1},\ldots,u_{1,\ell}\}$, $\{u_{2,1},\ldots,u_{2,\ell}\}$, $\ldots$, $\{u_{r,1},\ldots,u_{r,\ell}\}$ induce a complete $r$-partite graph. Let $G'$ be the graph induced by $\{u_0,u_{1,1},\ldots,u_{1,\ell}\}$, $\ldots$, $\{u_{r,1}, \ldots ,u_{r,\ell}\}$. Since $u_0$ is adjacent to $u_{1,1},\ldots,u_{1,a+1}$, then $e(G')> e(T_{r\ell+1,r})+a$, by the definition of Tur\'{a}n number, $G'$ contains an $F$, this is a contradiction. Therefore $u_0$ is adjacent to at most $a$ vertices in  $V_1\setminus (W\cup L)$.
Hence
\begin{eqnarray*}
d_{V_1}(u_0)&\le & |W|+|L|+a\\
&<&   \theta  n + \epsilon^{\frac{1}{3}} n +a \\
&<& 2\theta n,
\end{eqnarray*}
for sufficiently large $n$.  This is a contradiction to the fact that $u_0\in W$.
Hence $W\subseteq L$.

\end{proof}

\begin{lemma} \label{largeindependent}
For each $i\in [r]$,
$$e(G[V_i\setminus L])\leq a.$$ Furthermore, for each $i\in [r]$, there exists an independent set $I_i \subseteq V_i\setminus L$ such that $$|I_i| \geq |V_i|-  \epsilon^{\frac{1}{3}} n- a.$$
\end{lemma}
\begin{proof}  Suppose to the contrary that there exists an $i_0\in [r]$ such that $e(G[V_{i_0}\setminus L])> a.$ Without loss of generality, we may assume that $e(G[V_1\setminus L])> a.$
By Lemmas \ref{approximate structure} and \ref{lem size of W and L}, we have $|V_i\setminus L|\geq \left(\frac{1}{r}- 3\sqrt{\epsilon} \right)n -\epsilon^{\frac{1}{3}} n \ge  \ell$ for any $i \in [r]$. Let $u_{1,1},u_{1,2}, \ldots, u_{1,\ell} $ be $\ell$ vertices chosen   from $V_1 \setminus L$ such that the induced subgraph of $\{u_{1,1},u_{1,2}, \ldots, u_{1,\ell}\} $ in $G$ contains at least $a+1$ edges. For any $j\in [\ell]$, $u_{1,j}\notin L$ implies that $u_{1,j}\notin W$ by Lemma~\ref{lem size of W and L}, thus $d(u_{1,j})> \left(1- \frac{1}{r} - 3r\epsilon^{\frac{1}{3}} \right)n$, and $d_{V_1}(u_{1,j})< 2\theta n$. Then  we have
\begin{align}
d_{V_2}(u_{1,j})\nonumber
& \ge d(u_{1,j})-d_{V_1}(u_{1,j})- (r-2) \left(\frac{1}{r}+ 3\sqrt{\epsilon} \right)n  \\\label{degreeu2}
& > \frac{n}{r} - 3r\epsilon^{\frac{1}{3}} n - 2\theta n - 3(r-2)\sqrt{\epsilon} n  \\\nonumber
& > \frac{n}{r} - 6r\epsilon^{\frac{1}{3}} n - 2\theta n .
\end{align}
Applying Lemma~\ref{inclusion exclusion lemma}, we  get
    \begin{eqnarray*}
&& \left| N_{V_2}(u_{1,1})  \cap N_{V_2}(u_{1,2})\cap  \cdots \cap N_{V_2}(u_{1,\ell}) \setminus L \right| \\
 &\ge&\sum_{j=1}^{\ell} d_{V_2}(u_{1,j}) - (\ell-1) \left| V_2 \right|  - |L| \\
 &\ge&  \ell\left(\frac{n}{r} - 6r\epsilon^{\frac{1}{3}} n - 2\theta n  \right)
  - (\ell-1) \left(\tfrac{1}{r} + 3\sqrt{\epsilon} \right)n - \epsilon^{\frac{1}{3}} n \\
& > & \frac{n}{r} - 10r\ell\epsilon^{\frac{1}{3}} n-2\ell\theta n   > \ell,
\end{eqnarray*}
 for  sufficiently  large $n$. So there exist $\ell$ vertices $u_{2,1},u_{2,2}, \ldots, u_{2,\ell}\in V_2$
such that $\{u_{1,1}, \ldots, u_{1,\ell}\}$ and $\{u_{2,1},\ldots, u_{2,\ell}\}$ induce a complete bipartite graph. For an integer $s$ with $2\leq s\leq r-1$, suppose that for any $1\leq i\leq s$, there exist $u_{i,1},u_{i,2}, \ldots ,u_{i,\ell}\in V_i\setminus L$ such that $\{u_{1,1}, \ldots, u_{1,\ell}\}$, $\{u_{2,1},\ldots, u_{2,\ell}\}$, $\ldots$, $\{u_{s,1},\ldots, u_{s,\ell}\}$ induce
 a complete $s$-partite subgraph in $G$. We next consider the common neighbors of these vertices in $V_{s+1}$. Similarly, by  (\ref{degreeu2}),  we get that for each $i\in [s]$ and $j\in [\ell]$,
   \begin{align*}
    d_{V_{s+1}}(u_{i,j})
   & \ge \frac{n}{r} - 6r\epsilon^{\frac{1}{3}} n - 2\theta n .
  \end{align*}
By Lemma~\ref{inclusion exclusion lemma} again, we can obtain
\begin{eqnarray*}
&& \left| \left(\cap_{i\in [s],j\in [\ell]}N_{V_{s+1}}(u_{i,j}) \right)\setminus L \right| \\
&\ge&   \sum_{i\in [s], j\in [\ell]} d_{V_{s+1}}(u_{i,j})- (s\ell-1) \left| V_{s+1} \right|   - |L| \\
&\ge&   s\ell \left(\frac{n}{r} - 6r\epsilon^{\frac{1}{3}} n - 2\theta n\right)
  -  (s\ell-1) \left(\tfrac{1}{r} + 3\sqrt{\epsilon} \right)n  - \epsilon^{\frac{1}{3}} n \\
& > & \frac{n}{r} - 10rs\ell\epsilon^{\frac{1}{3}} n-2s\ell\theta n   > \ell,
\end{eqnarray*}
for sufficiently large $n$.
Thus there exist $\ell$ vertices $u_{s+1,1},u_{s+1,2}, \ldots ,u_{s+1,\ell}\in V_{s+1}\setminus L$ such that $\{u_{1,1}, \ldots, u_{1,\ell}\}$, $\{u_{2,1},\ldots, u_{2,\ell}\}$, $\ldots$, $\{u_{s+1,1},\ldots, u_{s+1,\ell}\}$ induce a complete $(s+1)$-partite subgraph in $G$. Therefore, for each $i \in [2,r]$, there exist $u_{i,1},u_{i,2}, \ldots ,u_{i,\ell}$ in $V_i \setminus L$ such that
$\{u_{1,1},\ldots,u_{1,\ell}\}$, $\{u_{2,1},\ldots,u_{2,\ell}\}$, $\ldots$, $\{u_{r,1},\ldots,u_{r,\ell}\}$ induce a complete $r$-partite graph. Let $G'$ be the graph induced by $\{u_{1,1},\ldots,u_{1,\ell}\}$, $\ldots$, $\{u_{r,1}, \ldots ,u_{r,\ell}\}$. Then $e(G')> e(T_{r\ell,r})+a$, which implies that $G'$ contains a copy of  $F$, this is a contradiction. Thus  for each $i \in [r]$, $e(G[V_i\setminus L])\leq a$.

  Therefore, the subgraph obtained from $G[V_i \setminus L]$ by deleting one vertex of each edge in $G[V_i\setminus L]$ contains no edges, which is an independent set of $G[V_i\setminus L]$. Therefore, for each $i \in [r]$, there exists an independent set
$I_i\subseteq V_i$ such that
\begin{align*}
|I_i| &\ge |V_i\setminus L|-a \ge |V_i|-  \epsilon^{\frac{1}{3}} n- a .
\end{align*}
\end{proof}

\begin{lemma}
\label{eigenvector entries}
$L$ is empty and $e(G[V_i])\leq a$ for each $i\in [r]$.
\end{lemma}

\begin{proof} We first prove that $L=\varnothing$.
Otherwise, let $v$ be a vertex in $L$. Then $d(v)\le (1- \frac{1}{r}-3r\epsilon^{\frac{1}{3}})n$. Recall that $x_z=\max\{x_i|\ i\in [n]\}$, then $\lambda(G)=\lambda (G) x_z=\sum_{wz\in E(G)}x_w\leq d(z)$. Hence
 $$d(z)\ge \lambda (G)\ge   \left(1- \frac{1}{r} - \frac{r}{4n^2}+\frac{2a}{n^2} \right)n > \left(1-\frac{1}{r}- 3r\epsilon^{\frac{1}{3}} \right)n,$$
 as $n$ is large enough. Hence $z\notin L$. Without loss of generality, we may assume that $z\in V_1$.
Let $G'$ be the graph with $V(G')=V(G)$ and edge set $E(G') = E(G \setminus \{v\}) \cup \{vw|\  w\in N(z)\cap(\cup_{i=2}^{r} I_i)\}$. 
We claim that $G'$ is $F$-free. Otherwise, $G'$ contains a copy of $F$, denoted  as $F'$, as a subgraph,  then $v\in V(F')$. Let $N_{G'}(v)\cap V(F')= \{w_1,\ldots, w_s\}$. Obviously, $w_i\notin V_1$ and $w_i\notin L$ for any $i\in [s]$. If $z\notin V(F')$,  then $(F'\setminus \{v\})\cup \{z\}$ is a copy of $F$ in $G$, which is a contradiction. Thus $z\in V(F')$. For any $i \in [s]$,
\begin{align*}
d_{V_1}(w_i)
&=d(w_i)-d_{V\setminus V_1}(w_i)\\
&\geq \left(1-\frac{1}{r}- 3r\epsilon^{\frac{1}{3}} \right)n-a-\epsilon^{\frac{1}{3}}n-(r-2)\left(\frac{n}{r}+3\sqrt{\epsilon}n\right)\\
&>\frac{n}{r}-6r\epsilon^{\frac{1}{3}}n-a,
\end{align*}
where the last second inequality holds as $w_i\notin L$ and $e(G[V_j\setminus L])\leq a$ for $w_i\in V_j$.
Using Lemma \ref{inclusion exclusion lemma}, we get
\begin{eqnarray*}
& & \Big|\bigcap_{i=1}^{s}N_{V_1}(w_{i})\setminus L\Big|\\[2mm]
&\geq &\sum_{i=1}^{s}d_{V_1}(w_{i})-(s-1)|V_1|- |L|\\[2mm]
&> & s\left(\frac{n}{r}-6r\epsilon^{\frac{1}{3}}n-a\right)-
(s-1)\left(\frac{n}{r}+3\sqrt{\epsilon}n\right)-\epsilon^{\frac{1}{3}}n\\[2mm]
&> & \frac{n}{r}-10sr\epsilon^{\frac{1}{3}}n -sa > 1.
\end{eqnarray*}
Thus there exists $v'\in V_1\setminus L$ such that $v'$ is adjacent to $w_{1},\ldots,w_{s}$. Then $(F'\setminus \{v\})\cup \{v'\}$ is a copy of  $F$ in $G$, which is a contradiction. Thus $G'$ is $F$-free.

   By Lemma \ref{largeindependent}, we have $e(G[V_1\setminus L])\leq a$, then the maximum degree in the induced subgraph $G[V_1\setminus L]$ is at most $a$. Combining this with Lemma~\ref{lem size of W and L}, we get 
 $$d_{V_1}(z)=d_{{V_1}\cap L}(z)+d_{V_1\setminus L}(z)\le \epsilon^{\frac{1}{3}} n + a.$$
 Therefore, by Lemma \ref{largeindependent}, we have
 \begin{eqnarray*}
 \lambda (G)&=&\lambda (G) x_z=\sum_{v\sim z} x_v
 = \sum_{\substack{v\in V_1 , v\sim z}} x_v
 + \sum_{i=2}^{r}\left(\sum_{\substack{v\in V_i  , v\sim z} } x_v\right)\\
 &=& \sum_{\substack{v\in V_1 , v\sim z}} x_v
 + \sum_{i=2}^{r}\left(\sum_{\substack{v\in I_i , v\sim z}} x_v
 + \sum_{\substack{v\in V_i\setminus I_i , v\sim z} } x_v\right)\\
 &\le&  d_{V_1}(z)+\sum_{i=2}^{r}\left(\sum_{\substack{v\in I_i, v\sim z}} x_v\right)
 + \sum_{i=2}^{r}|V_i\setminus I_i |\\
 &\le & \epsilon^{\frac{1}{3}} n + a +\sum_{i=2}^{r}\left(\sum_{\substack{v\in I_i, v\sim z}} x_v\right) + (r-1)(\epsilon^{\frac{1}{3}}n +a) .
\end{eqnarray*}
 By Lemma \ref{lem32}, we have
 \begin{equation}\label{Lempty2}
 \sum_{i=2}^{r}\left(\sum_{\substack{v\in I_i, v\sim z}} x_v\right)
 \ge \left(1- \frac{1}{r}\right)n - \frac{r}{4n}+\frac{2a}{n} - r\epsilon^{\frac{1}{3}} n - ra.
 \end{equation}
By the Rayleigh quotient equation,
\begin{align*}
\lambda(G') - \lambda(G)
&\geq \frac{\mathbf{x}^T\left(A(G')-A(G)\right)\mathbf{x}}{\mathbf{x}^T\mathbf{x}}
 = \frac{2x_v}{\mathbf{x}^T\mathbf{x}}\left(
\sum_{i=2}^{r}\left(\sum_{\substack{w\in I_i, v\sim z}} x_w\right) - \sum_{uv\in E(G)} x_u\right) \\
& \geq\frac{2x_v}{\mathbf{x}^T\mathbf{x}}
\left( \left(1- \frac{1}{r}\right)n - \frac{r}{4n}+\frac{2a}{n} - r\epsilon^{\frac{1}{3}} n - ra - \left(1- \frac{1}{r}-3r\epsilon^{\frac{1}{3}}\right)n \right)>0,
\end{align*}
where the last second inequality holds since  (\ref{Lempty2}) and $\sum_{uv\in E(G)} x_u\leq d(v)$, and the last inequality holds for $n$ large enough. This contradicts the fact that $G$ has the largest spectral radius over all $F$-free graphs, so $L$ must be empty.
Furthermore, by Lemma \ref{largeindependent}, we have $e(G[V_i])\leq a$ for each $i\in [r]$.
\end{proof}

\begin{lemma}\label{Bi}
For any $i\in [r]$, let $B_i=\{u\in V_i|\ d_{V_i}(u)\geq 1\}$ and $C_i=V_i\setminus B_i$. Then

(1) $|B_i|\leq 2a$;

(2) For every vertex $u\in C_i$, $u$ is adjacent to all vertices of $V\setminus V_i$.
\end{lemma}
\begin{proof}
We prove the  assertions by contradiction.

(1) If there exists a $j\in [r]$ such that $|B_j|> 2a$, then $\sum_{u\in B_j}d_{V_j}(u)>2a$. On the other hand,  $e(G[V_j])\leq a$ by Lemma \ref{eigenvector entries} . Therefore,
\[
2a< \sum_{u\in B_j}d_{V_j}(u)=\sum_{u\in V_j}d_{V_j}(u)=2e(G[V_j])\leq 2a,
\]
which is a contradiction.

(2) If there exists a vertex $v\in C_{i_0}$ such that there is a vertex $w_1\notin V_{i_0}$ and $vw_{1}\notin E(G)$, where $i_0\in [r]$. Let $G'$ be the graph with $V(G')=V(G)$ and $E(G')=E(G)\cup \{vw_{1}\}$. We claim that $G'$ is $F$-free. Otherwise, $G'$ contains a copy of $F$, denoted as $F'$,  as a subgraph, then $vw_{1}\in E(F')$. Let $N_{G'}(v)\cap V(F')= \{w_1,\ldots, w_s\}$. Obviously, $w_i\notin V_{i_0}$ for any $i\in [s]$, then we have,
\begin{align}
d_{V_{i_0}}(w_i)&=d(w_i)-d_{V\setminus V_{i_0}}(w_i) \nonumber \\
&\geq \left(1-\frac{1}{r}- 3r\epsilon^{\frac{1}{3}} \right)n-a-(r-2)\left(\frac{n}{r}+3\sqrt{\epsilon}n\right)\label{degree}\\
&>\frac{n}{r}-6r\epsilon^{\frac{1}{3}}n-a, \nonumber
\end{align}
where the last second inequality holds as $L=\emptyset$, and $e(G[V_j])\leq a$ for $w_i\in V_j$. Using Lemma \ref{inclusion exclusion lemma}, we consider the common neighbors of $w_1,\ldots, w_s$ in $C_{i_0}$,
\begin{eqnarray*}
& & \Big|\bigcap_{i=1}^{s}N_{V_{i_0}}(w_{i})\setminus B_{i_0}\Big|\\[2mm]
&\geq &\sum_{i=1}^{s}d_{V_{i_0}}(w_{i})-(s-1)|V_{i_0}|- |B_{i_0}|\\[2mm]
&> & s\left(\frac{n}{r}-6r\epsilon^{\frac{1}{3}}n-a\right)-
(s-1)\left(\frac{n}{r}+3\sqrt{\epsilon}n\right)-2a\\[2mm]
&>& \frac{n}{r}-9rs\epsilon^{\frac{1}{3}}n - (s+2)a > 1.
\end{eqnarray*}
Then there exists $v'\in C_{i_0}$ such that $v'$ is adjacent to $w_{1},\ldots,w_{s}$.
 Then $(F'\setminus \{v\})\cup \{v'\}$ is a copy of  $F$ in $G$, which is a contradiction. Thus $G'$ is $F$-free. From the construction of $G'$, we see that $\lambda(G')>\lambda(G)$, which contradicts the assumption that $G$ has the maximum spectral radius among all $F$-free graphs on $n$ vertices.

\end{proof}

\begin{lemma}\label{eigenvector}
For any $u\in V(G)$,  $x_u\geq 1-\frac{20a^2r^2}{n}$.
\end{lemma}

\begin{proof}

We will prove this lemma by contradiction. Suppose that there is a vertex $v\in V(G)$ with $x_v< 1-\frac{20a^2r^2}{n}$.
Recall that $x_z=\max\{x_i|\ i\in V(G)\}=1$. Without loss of generality, we may assume that $z\in V_1$.
Let $G'$ be the graph with $V(G')=V(G)$ and $E(G')=E(G\setminus \{v\})\cup \{vw|\ w\in N(z)\cap (\cup_{i=2}^{r}C_i)\}$. We claim that
 $G'$ is $F$-free.
Otherwise, $G'$ contains a copy of $F$, denoted by $F'$, as a subgraph, then $v\in V(F')$. Let $N_{G'}(v)\cap V(F')= \{w_1,\ldots, w_s\}$. Obviously, $w_i\notin V_1$ for any $i\in [s]$. If $z\notin V(F')$, then  $(F'\setminus \{v\})\cup \{z\}$ is a copy of $F$ in $G$, which is a contradiction. Thus $z\in V(F')$. By using the similar method as in Lemma \ref{Bi}, we  get
\begin{align*}
d_{V_1}(w_i)>\frac{n}{r}-6r\epsilon^{\frac{1}{3}}n-a,
\end{align*}
 for any $i \in [s]$.
Using Lemma \ref{inclusion exclusion lemma}, we consider the common neighbors of $w_1,\ldots, w_s$ in $C_1$,
\begin{eqnarray*}
& & \Big|\bigcap_{i=1}^{s}N_{V_1}(w_{i})\setminus B_1\Big|\\[2mm]
&\geq &\sum_{i=1}^{s}d_{V_1}(w_{i})-(s-1)|V_1|- | B_1|\\[2mm]
&> & s\left(\frac{n}{r}-6r\epsilon^{\frac{1}{3}}n-a\right)-
(s-1)\left(\frac{n}{r}+3\sqrt{\epsilon}n\right)-2a\\[2mm]
&>& \frac{n}{r}-9rs\epsilon^{\frac{1}{3}}n - (s+2)a > 1.
\end{eqnarray*}
Then there exists $v'\in C_1$ such that $v'$ is adjacent to $w_{1},\ldots,w_{s}$.
 Then $(F'\setminus \{v\})\cup \{v'\}$ is a copy of  $F$ in $G$, which is a contradiction. Thus $G'$ is $F$-free.

By Lemma \ref{eigenvector entries}, $e(G[V_1])\leq a$, then $d_{V_1}(z)\leq a$. By (\ref{eigenequation}), we have
\begin{align*}
\lambda (G) x_z&=\sum_{w\thicksim z}x_w = \sum_{w\thicksim z,w\in V_1}x_w+\sum_{i=2}^{r}\Big(\sum_{w\thicksim z,w\in V_i}x_w\Big)\\
&= \sum_{w\thicksim z,w\in V_1}x_w+\sum_{i=2}^{r}\Big(\sum_{w\thicksim z,w\in B_i}x_w+\sum_{w\thicksim z,w\in C_i}x_w\Big),
\end{align*}
which implies that
\begin{align}
\sum_{i=2}^{r}\Big(\sum_{w\thicksim z,w\in C_i}x_w\Big)
&=\lambda (G)-\sum_{w\thicksim z,w\in V_1}x_w-\sum_{i=2}^{r}\Big(\sum_{w\thicksim z,w\in B_i}x_w\Big)\nonumber\\
&\geq \lambda (G)-d_{V_1}(z)-\sum_{i=2}^{r}\Big(\sum_{w\in B_i}1\Big)\nonumber\\[2mm]
&\geq \lambda (G) -a-(r-1)2a,\label{G'}\\
&= \lambda (G)-(2r-3)a\nonumber,
\end{align}
where  (\ref{G'}) holds as $e(G[V_1])\leq a$, and $|B_i|\leq 2a$ for any $i\in [r]$.

By Rayleigh quotient equation, we have
\begin{align*}
\lambda(G')-\lambda(G)&\geq \frac{\mathbf{x}^{T}(A(G')-A(G))\mathbf{x}}{\mathbf{x}^T\mathbf{x}}\\[2mm]
&= \frac{2x_v}{\mathbf{x}^T\mathbf{x}}\left(\sum_{i=2}^{r}\Bigl(\sum_{w\thicksim z,w\in C_i}x_w \Bigr)-\sum_{uv\in E(G)}x_u\right)\\[2mm]
&= \frac{2x_v}{\mathbf{x}^T\mathbf{x}}\left(\sum_{i=2}^{r}\Bigl(\sum_{w\thicksim z,w\in C_i}x_w \Bigr)-\lambda (G) x_v\right)\\[2mm]
&>  \frac{2x_v}{\mathbf{x}^T\mathbf{x}}\left(\lambda (G)-(2r-3)a-\lambda (G)\Bigl(1-\frac{20a^2r^2}{n}\Bigr)\right)\\[2mm]
&\geq  \frac{2x_v}{\mathbf{x}^T\mathbf{x}}
\left(\frac{r-1}{r}20a^2r^2-\frac{r}{4n}\frac{20a^2r^2}{n}+\frac{2a}{n}\frac{20a^2r^2}{n}- (2r-3)a\right) >0,
\end{align*}
where the last second inequality holds as  (\ref{G'}), and  the last  inequality follows by
$\lambda (G) \ge \left(1- \frac{1}{r}\right)n - \frac{r}{4n}+\frac{2a}{n}$.
 This contradicts the assumption that $G$ has the maximum spectral radius among all $F$-free graphs on $n$ vertices. Thus $x_u\geq 1-\frac{20a^2r^2}{n}$ for any $u\in V(G)$.

\end{proof}

Let $G_{in}=\cup_{i=1}^{r}G[V_i]$. For any $i\in[r]$, let $|V_i|=n_i$, $K=K_{r}(n_1,n_2,\ldots,n_{r})$ be the complete $r$-partite graph on $V_1,V_2,\ldots,V_{r}$, and $G_{out}$ be the graph with $V(G_{out})=V(G)$ and $E(G_{out})=E(K)\setminus E(G)$.

\begin{lemma}\label{in}
$e(G_{in})-e(G_{out})\leq a.$
\end{lemma}
\begin{proof}
Suppose to the contrary that $e(G_{in})-e(G_{out})> a.$ For each $i\in [r]$, let $S_i$ be the vertex set satisfying
$B_i\subseteq S_i\subseteq V_i$ and $|S_i|=\ell$. Let $S=\cup_{i=1}^{r}S_i$, $G'=G[S]$. By Lemma \ref{Bi}, we have
$e(G')\geq  e(T_{r\ell,r})+e(G_{in})-e(G_{out})>e(T_{r\ell,r})+a,$
 which implies that $G'$ contains an $F$, this is a contradiction. So $e(G_{in})-e(G_{out})\leq a.$

\end{proof}

\begin{lemma}\label{balance}
For any $1\leq i<j\leq r$,  $\left|n_i-n_j\right|\leq 1$.
\end{lemma}

\begin{proof}
 We prove this lemma  by contradiction. 
 Without loss of generality, suppose that $n_1\geq n_2\geq \ldots \geq n_{r}$. Assume that there exist $i_0, j_0$ with $1\leq i_0 < j_0\leq r$ such that $n_{i_0}-n_{j_0}\geq 2$.

\noindent{\bfseries Claim 1.} There exists a constant $c_1>0$ such that $\lambda(T_{n,r})-\lambda(K)\geq \frac{c_1}{n}$.
\begin{proof}
Let $K'=K_r(n_1,\ldots, n_{i_0}-1,\ldots,n_{j_0}+1,\ldots,n_{r}).$ Assume $K'\cong K_r(n'_1,n'_2,\ldots,n'_{r})$, where $n'_1\geq n'_2\geq \ldots \geq n'_{r}$.
By (\ref{eq0}), we have
\begin{equation}\label{C}
1=\sum_{i=1}^{r}\frac{n_i}{\lambda(K)+n_i}=\frac{n_{i_0}}{\lambda(K)+n_{i_0}}+\frac{n_{j_0}}{\lambda(K)+n_{j_0}}+\sum_{i\in [r]\setminus \{i_0,j_0\}}\frac{n_i}{\lambda(K)+n_i},
\end{equation}
and
\begin{equation}\label{C'}
1=\sum_{i=1}^{r}\frac{n'_i}{\lambda(K')+n'_i}=\frac{n_{i_0}-1}{\lambda(K')+n_{i_0}-1}+\frac{n_{j_0}+1}{\lambda(K')+n_{j_0}+1}+\sum_{i\in [r]\setminus\{i_0,j_0\}}\frac{n_i}{\lambda(K')+n_i}.
\end{equation}

Subtracting  (\ref{C'}) from   (\ref{C}), we get
\begin{eqnarray*}
& &\frac{2(n_{i_0}-n_{j_0}-1)\lambda^2(K)+(n_{i_0}+n_{j_0})(n_{i_0}-n_{j_0}-1)\lambda(K)}{(\lambda(K)+n_{i_0}-1)(\lambda(K)+n_{i_0})(\lambda(K)+n_{j_0}+1)(\lambda(K)+n_{j_0})}\\[2mm]
& = & \sum_{i\in [r]\setminus\{i_0,j_0\}}\frac{n_i(\lambda(K')-\lambda(K))}{(\lambda(K)+n_i)(\lambda(K')+n_i)}+\frac{(n_{i_0}-1)(\lambda(K')-\lambda(K))}{(\lambda(K)+n_{i_0}-1)(\lambda(K')+n_{i_0}-1)}\\[2mm]
& &+\frac{(n_{j_0}+1)(\lambda(K')-\lambda(K))}{(\lambda(K)+n_{j_0}+1)(\lambda(K')+n_{j_0}+1)}\\[2mm]
& \leq & \frac{\lambda(K')-\lambda(K)}{\lambda(K)+n'_{r}}\Big(\sum_{i\in [r]\setminus\{i_0,j_0\}}\frac{n_i}{\lambda(K')+n_i}+\frac{n_{i_0}-1}{\lambda(K')+n_{i_0}-1}+\frac{n_{j_0}+1}{\lambda(K')+n_{j_0}+1}\Big)\\[2mm]
& = & \frac{\lambda(K')-\lambda(K)}{\lambda(K)+n'_{r}},
\end{eqnarray*}
where the inequality holds as $n'_{r}\leq \min\{n_1,\ldots,n_{i_0}-1,\ldots,n_{j_0}+1,\ldots,n_{r}\}$, and the last equality holds by  (\ref{C'}). Combining with the assumption $n_{i_0}-n_{j_0}\geq 2$, we obtain
\begin{eqnarray}\label{C-C'}
\frac{2\lambda^2(K)+(n_{i_0}+n_{j_0})\lambda(K)}{(\lambda(K)+n_{i_0}-1)(\lambda(K)+n_{i_0})(\lambda(K)+n_{j_0}+1)(\lambda(K)+n_{j_0})}\leq \frac{\lambda(K')-\lambda(K)}{\lambda(K)+n'_{r}}.
\end{eqnarray}
In view of  the construction of $K$, we see that
$$n-\Big(\frac{n}{r}+3\sqrt{\epsilon} n\Big) \leq \delta(K)\leq \lambda(K)\leq \Delta(K)\leq n-\Big(\frac{n}{r}-3\sqrt{\epsilon} n\Big),$$ thus $\lambda(K)=\Theta(n)$. From  (\ref{C-C'}), it follows  that there exists a constant $c_1>0$ such that $\lambda(K')-\lambda(K)\geq \frac{c_1}{n}$. Therefore, by Lemma \ref{rpartitegraph}, $\lambda(T_{n,r})-\lambda(K)\geq\lambda(K')-\lambda(K)\geq \frac{c_1}{n}$.

\end{proof}

\noindent{\bfseries Claim 2.} There exists a constant $c_2>0$ such that $\lambda(T_{n,r})-\lambda(K)\leq \frac{c_2}{n^2}$.

\begin{proof}
According to the definition of $K$, we have $e(G)=e(G_{in})+e(K)-e(G_{out})$.
By Lemma \ref{Bi}, for any $i\in [r]$, and every vertex $u\in C_i$, $u$ is adjacent to all vertices of $V\setminus V_i$. Thus
 $$e(G_{out})\leq \sum_{1\leq i<j\leq r}|B_i||B_j|\leq \binom{r}{2}(2a)^2\leq 2a^2r^2.$$ Therefore
\begin{align}
\lambda(G)&= \frac{\mathbf{x}^\mathrm{T}A(G)\mathbf{x}}{\mathbf{x}^{\mathrm{T}}\mathbf{x}}\nonumber\\[2mm]
&= \frac{2\sum_{ij\in E(K)}x_ix_j}{\mathbf{x}^{\mathrm{T}}\mathbf{x}}+ \frac{2\sum_{ij\in E(G_{in})}x_ix_j}{\mathbf{x}^{\mathrm{T}}\mathbf{x}}- \frac{2\sum_{ij\in E(G_{out})}x_ix_j}{\mathbf{x}^{\mathrm{T}}\mathbf{x}} \nonumber\\[2mm]
&\leq \lambda(K)+ \frac{2e(G_{in})}{\mathbf{x}^{\mathrm{T}}\mathbf{x}}- \frac{2e(G_{out})(1-\frac{20a^2r^2}{n})^2}{\mathbf{x}^{\mathrm{T}}\mathbf{x}} \nonumber\\[2mm]
&\leq \lambda(K)+\frac{2(e(G_{in})-e(G_{out}))}{\mathbf{x}^{\mathrm{T}}\mathbf{x}}+\frac{e(G_{out})\frac{40a^2r^2}{n}}{\mathbf{x}^{\mathrm{T}}\mathbf{x}}\nonumber\\[2mm]
&\leq \lambda(K)+\frac{2a}{\mathbf{x}^{\mathrm{T}}\mathbf{x}}+\frac{\frac{80a^4r^4}{n}}{\mathbf{x}^{\mathrm{T}}\mathbf{x}},\label{GC}
\end{align}
where  (\ref{GC}) holds by Lemma \ref{in} and $e(G_{out})\leq  2a^2r^2.$

On the other hand, let $\mathbf{y}$ be an eigenvector of $T_{n,r}$ corresponding to $\lambda(T_{n,r})$,
$k=n-r\lfloor\frac{n}{r}\rfloor$.
Since $T_{n,r}$ is a complete $r$-partite graph on $n$ vertices where
each partite set has either $\lfloor\frac{n}{r}\rfloor$ or $\lceil\frac{n}{r}\rceil$ vertices,  we may assume $\mathbf{y}=(\underbrace{y_1,\ldots,y_1}_{k\lceil\frac{n}{r}\rceil},\underbrace{y_2,\ldots,y_2}_{n-k\lceil\frac{n}{r}\rceil})^{\mathrm{T}}$.
Thus by  (\ref{eigenequation}), we have
\begin{eqnarray}
\lambda(T_{n,r})y_1=(r-k)\big\lfloor\frac{n}{r}\big\rfloor y_2+(k-1)\big\lceil\frac{n}{r}\big\rceil y_1,\label{14}
\end{eqnarray}
and \begin{eqnarray}
\lambda(T_{n,r})y_2=(r-k-1)\big\lfloor\frac{n}{r}\big\rfloor y_2+k\big\lceil\frac{n}{r}\big\rceil y_1.\label{15}
\end{eqnarray}
Combining (\ref{14}) and (\ref{15}), we obtain
\[
\Big(\lambda(T_{n,r})+\big\lceil\frac{n}{r}\big\rceil\Big)
y_1=\Big(\lambda(T_{n,r})+\big\lfloor\frac{n}{r}\big\rfloor\Big)
y_2.
\]
Without loss of generality, we assume that $y_2=1$. Then
\[
y_2\geq y_1=\frac{\lambda(T_{n,r})+\lfloor\frac{n}{r}\rfloor}{\lambda(T_{n,r})+\lceil\frac{n}{r}\rceil}\geq 1-\frac{1}{\lambda(T_{n,r})+\lceil\frac{n}{r}\rceil}.
\]
Since $\lambda(T_{n,r})\geq \delta(T_{n,r})\geq  n-\lceil\frac{n}{r}\rceil$,  $y_1\geq 1-\frac{1}{n}$.
Let $H\in \mathrm{Ex}(n,F)$. Then $e(H)=\mathrm{ex}(n,F)=e(T_{n,r})+a$. Therefore
\begin{align}
\lambda(G)&\geq \lambda(H)\nonumber\geq \frac{\mathbf{y}^\mathrm{T}A(H)\mathbf{y}}{\mathbf{y}^{\mathrm{T}}\mathbf{y}}\nonumber\\[2mm]
&\geq \frac{\mathbf{y}^\mathrm{T}A(T_{n,r})\mathbf{y}}{\mathbf{y}^{\mathrm{T}}\mathbf{y}}+\frac{2a}{\mathbf{y}^{\mathrm{T}}\mathbf{y}}\left(1-\frac{1}{n}\right)^2\nonumber\\[2mm]
&\geq \lambda(T_{n,r})+\frac{2a}{n}\left(1-\frac{2}{n}\right).\label{Tnr-1}
\end{align}
Combining   (\ref{GC}), (\ref{Tnr-1}) and $\mathbf{x}^{\mathrm{T}}\mathbf{x}\geq n(1-\frac{20a^2r^2}{n})^2\geq n-40a^2r^2$, we get
\begin{eqnarray*}
& &\lambda(T_{n,r})-\lambda(K)\\[2mm]
&\leq & \frac{2a}{\mathbf{x}^{\mathrm{T}}\mathbf{x}}-\frac{2a}{n}+\frac{4a}{n^2}+\frac{\frac{80a^4r^4}{n}}{\mathbf{x}^{\mathrm{T}}\mathbf{x}}\\[2mm]
&\leq & \frac{2a}{n-40a^2r^2}-\frac{2a}{n}+\frac{4a}{n^2}+\frac{\frac{80a^4r^4}{n}}{n-40a^2r^2}\\[2mm]
&\leq & \frac{80a^3r^2}{n(n-40a^2r^2)}+\frac{4a}{n^2}+\frac{80a^4r^4}{n(n-40a^2r^2)}\\[2mm]
&\leq & \frac{c_2}{n^2},
\end{eqnarray*}
where $c_2$ is a positive constant.
\end{proof}
 Combining Claim 1 and Claim 2, we have
 \[
 \frac{c_1}{n}\leq \lambda(T_{n,r})-\lambda(K)\leq \frac{c_2}{n^2},
 \]
 which is a contradiction when $n$ is sufficiently large. Thus $\left|n_i-n_j\right|\leq 1$ for any $1\leq i<j\leq r$.

\end{proof}

\medskip
\noindent{\bfseries Proof of Theorem \ref{main result}.} Now we  prove that $e(G)=\mathrm{ex}(n,F)$. Otherwise, we assume that $e(G)\leq \mathrm{ex}(n,F)-1$. Let $H\in \mathrm{Ex}(n,F)$. Then $|E(H)|=e(T_{n,r})+a$. By Lemma \ref{balance}, we may assume that $V_1\cup \ldots \cup V_{r}$ is a vertex partition of $H$. Let $E_{1}=E(G)\setminus E(H)$, $E_{2}=E(H)\setminus E(G)$,  then $E(H)=(E(G)\cup E_{2})\setminus E_{1}$, and
\[
|E(G)\cap E(H)|+|E_{1}|=e(G)<e(H)=|E(G)\cap E(H)|+|E_{2}|,
\]
which implies that
$|E_{2}|\geq |E_{1}|+1$. Furthermore, by Lemma \ref{Bi}, we have
\begin{eqnarray}
|E_{2}|\leq a+ \sum\limits_{1\leq i<j\leq r}|B_i||B_j|\leq a+\binom{r}{2}(2a)^2\leq 3a^2r^2.\label{eqn3}
\end{eqnarray}
According to  (\ref{Rayleigh}) and (\ref{eqn3}), for sufficiently large $n$, we have
\begin{align*}
\lambda(H)&\geq \frac{\mathbf{x}^{\mathrm{T}}A(H)\mathbf{x}}{\mathbf{x}^{\mathrm{T}}\mathbf{x}}\\[2mm]
&= \frac{\mathbf{x}^{\mathrm{T}}A(G)\mathbf{x}}{\mathbf{x}^{\mathrm{T}}\mathbf{x}}+\frac{2\sum_{ij\in E_{2}}x_ix_j}{\mathbf{x}^{\mathrm{T}}\mathbf{x}}-\frac{2\sum_{ij\in E_{1}}x_ix_j}{\mathbf{x}^{\mathrm{T}}\mathbf{x}}\\[2mm]
&= \lambda(G)+\frac{2}{\mathbf{x}^{\mathrm{T}}\mathbf{x}}\Big(\sum_{ij\in E_{2}}x_ix_j-\sum_{ij\in E_{1}}x_ix_j\Big)\\[2mm]
&\geq  \lambda(G)+\frac{2}{\mathbf{x}^{\mathrm{T}}\mathbf{x}}\Big(|E_{2}|(1-\frac{20a^2r^2}{n})^2- |E_{1}|\Big)\\[2mm]
&\geq  \lambda(G)+\frac{2}{\mathbf{x}^{\mathrm{T}}\mathbf{x}}\Big(|E_{2}|-\frac{40a^2r^2}{n}|E_2|- |E_{1}|\Big)\\[2mm]
&\geq \lambda(G)+\frac{2}{\mathbf{x}^{\mathrm{T}}\mathbf{x}}\Big(1-\frac{40a^2r^2}{n}|E_{2}|\Big)\\[2mm]
&\geq \lambda(G)+\frac{2}{\mathbf{x}^{\mathrm{T}}\mathbf{x}}\Big(1-\frac{40a^2r^2}{n}3a^2r^2\Big)\\[2mm]
&> \lambda(G),
\end{align*}
which contradicts the assumption that $G$ has the maximum spectral radius among all $F$-free graphs on $n$ vertices. Hence $e(G)=\mathrm{ex}(n,F)$.
\qed


\end{document}